\documentclass[12pt]{amsart}
\usepackage{amscd, amssymb, pgf, tikz, hyperref}

\newcommand{\field}[1]{\mathbb{#1}}
\newcommand{\CC}{\field{C}}
\newcommand{\NN}{\field{N}}

\newcommand{\TT}{\field{T}}
\newcommand{\ZZ}{\field{Z}}

\newcommand{\Aa}{\mathcal{A}}
\newcommand{\Bb}{\mathcal{B}}

\newcommand{\Dd}{\mathcal{D}}
\newcommand{\Ee}{\mathcal{E}}
\newcommand{\Ff}{\mathcal{F}}
\newcommand{\Gg}{\mathcal{G}}
\newcommand{\Hh}{\mathcal{H}}
\newcommand{\Kk}{\mathcal{K}}
\newcommand{\Ll}{\mathcal{L}}
\newcommand{\Mm}{\mathcal{M}}
\newcommand{\Nn}{\mathcal{N}}
\newcommand{\Oo}{\mathcal{O}}
\newcommand{\Rr}{\mathcal{R}}

\newcommand{\Tt}{\mathcal{T}}

\newcommand{\la}{\langle}
\newcommand{\ra}{\rangle}
\newcommand{\ve}{\varepsilon}
\newcommand{\ds}{\displaystyle}

\newtheorem{thm}{Theorem}[section]
\newtheorem{cor}[thm]{Corollary}
\newtheorem{lem}[thm]{Lemma}
\newtheorem{prop}[thm]{Proposition}

\theoremstyle{definition}
\newtheorem{dfn}[thm]{Definition}

\theoremstyle{remark}
\newtheorem{rmk}[thm]{Remark}

\newtheorem{example}[thm]{Example}

\newtheorem*{examples*}{Examples}

\numberwithin{equation}{subsection}

\title{$C^*$-algebras  from  $k$ group representations}

\author{Valentin Deaconu}
\address{Valentin Deaconu \\ Department of Mathematics \& Statistics\\ University
of Nevada\\ Reno NV 89557-0084\\ USA} \email{vdeaconu@unr.edu}



\begin{document}
\begin{abstract}We introduce certain $C^*$-algebras and $k$-graphs associated to $k$ finite dimensional unitary representations $\rho_1,...,\rho_k$ of a compact group $G$. We define a higher rank Doplicher-Roberts algebra $\Oo_{\rho_1,...,\rho_k}$, constructed from intertwiners of tensor powers of  these representations.
Under certain conditions, we show that this $C^*$-algebra is  isomorphic to a corner in the $C^*$-algebra of  a row finite rank $k$ graph $\Lambda$ with no sources. For $G$ finite and $\rho_i$ faithful of dimension at least $2$, this graph is irreducible, it has vertices $\hat{G}$ and the edges are determined by $k$ commuting matrices obtained from the character table of the group.
We illustrate with some examples when $\Oo_{\rho_1,...,\rho_k}$ is simple and purely infinite, and with some $K$-theory computations.
\end{abstract}
\maketitle
\section{introduction}

\bigskip

The study of graph $C^*$-algebras was motivated among other reasons by the Doplicher-Roberts algebra $\Oo_\rho$ associated to a group representation $\rho$, see \cite{MRS, KPRR}. It is natural to imagine that a rank $k$ graph is related to a fixed set of $k$ representations $\rho_1,...,\rho_k$  satisfying certain properties.  

Given a  compact group $G$ and $k$ finite dimensional unitary representations $\rho_i$ on Hilbert spaces $\mathcal H_i$ of dimensions $d_i$ for $i=1,...,k$, we first construct a product system $\mathcal E$ indexed by the semigroup $(\NN^k,+)$ with fibers $\mathcal E_{n}=\mathcal H_1^{\otimes n_1}\otimes\cdots \otimes\mathcal H_k^{\otimes n_k}$ for $n=(n_1,...,n_k)\in \NN^k$.  Using the representations $\rho_i$,  the group $G$ acts on each fiber of $\Ee$ in a compatible way, so we obtain an action of $G$ on the Cuntz-Pimsner algebra $\Oo(\Ee)$. This action determines the crossed product $\Oo(\Ee)\rtimes G$ and the fixed point algebra $\Oo(\Ee)^G$. 

 Inspired from Section 7 of \cite{KPRR} and Section 3.3 of \cite{AM}, we  define a higher rank Doplicher-Roberts algebra $\mathcal O_{\rho_1,...,\rho_k}$ associated to the representations $\rho_1,...,\rho_k$. This algebra is constructed from intertwiners $Hom (\rho^n, \rho^m)$, where $\rho^n=\rho_1^{\otimes n_1}\otimes\cdots\otimes  \rho_k^{\otimes n_k}$ acting on $\Hh^n=\mathcal H_1^{\otimes n_1}\otimes\cdots\otimes \mathcal H_k^{\otimes n_k}$ for $n=(n_1,...,n_k)\in \mathbb N^k$. We show that $\mathcal O_{\rho_1,...,\rho_k}$ is isomorphic to $\Oo(\Ee)^G$.

If the representations $\rho_1,...,\rho_k$ satisfy some mild conditions,  we  construct  a $k$-coloured graph $\Lambda$ with vertex space $\Lambda^0=\hat{G}$, and with  edges $\Lambda^{\ve_i}$ given by some matrices $M_i$ indexed by $\hat{G}$. Here $\ve_i=(0,...,1,...,0)\in\NN^k$ with $1$ in position $i$ are the canonical generators. The matrices $M_i$ have entries \[M_i(w,v)=|\{e\in \Lambda^{\ve_i}: s(e)=v, r(e)=w\}|=\dim Hom(v,w\otimes \rho_i),\] the multiplicity of $v$ in $w\otimes \rho_i$ for  $i=1,...,k$. 
The matrices $M_i$ commute because $\rho_i\otimes\rho_j\cong \rho_j\otimes \rho_i$ for all $i,j=1,...,k$ and therefore \[\dim Hom(v,w\otimes \rho_i\otimes\rho_j)=\dim Hom(v,w\otimes \rho_j\otimes\rho_i).\] 

By a particular choice of isometric intertwiners in $Hom(v,w\otimes \rho_i)$ for each $v,w\in \hat{G}$ and for each $i$, we can choose  bijections \[\lambda_{ij}:\Lambda^{\ve_i}\times_{\Lambda^0}\Lambda^{\ve_j}\to \Lambda^{\ve_j}\times_{\Lambda^0}\Lambda^{\ve_i},\] obtaining a set of commuting squares for $\Lambda$. 
For $k\ge 3$, we need to check the associativity of the commuting squares,  i.e.
\[(id_\ell\times \lambda_{ij})(\lambda_{i\ell}\times id_j)(id_i\times \lambda_{j\ell})=(\lambda_{j\ell}\times id_i)(id_j\times \lambda_{i\ell})(\lambda_{ij}\times id_\ell)\]
as bijections from $\Lambda^{\ve_i}\times_{\Lambda^0}\Lambda^{\ve_j}\times_{\Lambda^0}\Lambda^{\ve_\ell}$ to $\Lambda^{\ve_\ell}\times_{\Lambda^0}\Lambda^{\ve_j}\times_{\Lambda^0}\Lambda^{\ve_i}$ for all $i<j<\ell$, see \cite{FS}. If these conditions are satisfied, we obtain a rank $k$ graph $\Lambda$, which is row-finite with no sources, but in general not unique.  

In many situations,  $\Lambda$ is cofinal and it satisfies the aperiodicity condition, so $C^*(\Lambda)$ is simple. For $k=2$, the $C^*$-algebra $C^*(\Lambda)$ is unique when it is simple and purely infinite, because   its $K$-theory depends only on the matrices $M_1, M_2$. It is an open question what happens for $k\ge 3$.

Assuming that the representations $\rho_1,...,\rho_k$ determine a rank $k$ graph $\Lambda$, we prove that the Doplicher-Roberts algebra $\mathcal O_{\rho_1,...,\rho_k}$ is isomorphic to a corner of $C^*(\Lambda)$, so if $C^*(\Lambda)$ is simple, then $\mathcal O_{\rho_1,...,\rho_k}$ is Morita equivalent to $C^*(\Lambda)$. In particular cases we can compute its $K$-theory using results from \cite{E}.

 \bigskip
\section{The product system}

\bigskip
 Product systems over arbitrary semigroups were introduced by N. Fowler \cite{ F}, inspired by work of W. Arveson, and studied by several authors, see \cite{SY, CLSV, AM}. In this paper, we will mostly be interested  in  product systems $\Ee$ indexed by $( \NN^k , +)$, associated to some representations $\rho_1,...,\rho_k$ of a compact group $G$.
We remind some general definitions and constructions with product systems, but we will consider the Cuntz-Pimsner algebra $\Oo(\Ee)$ and we will mention some properties only in particular cases.
\begin{dfn}
Let $(P, \cdot)$ be a discrete semigroup with identity $e$ and let $A$ be a $C^*$-algebra.
A {\em product system}  of  $C^*$-correspondences over $A$ indexed by $P$
is a semigroup $\Ee=\bigsqcup_{p\in P}\Ee_p$ and a map $\Ee\to P$ such that
\begin{itemize}
\item for each $p\in P$, the fiber $\Ee_p\subset \Ee$ is a  $C^*$-correspondence over $A$
with inner product $\langle\cdot,\cdot\rangle_p$;
\item the identity fiber $\Ee_e$ is $A$ viewed as a $C^*$-correspondence over itself;
\item for $p,q\in P\setminus\{e\}$ the multiplication map \[\Mm_{p,q}:\Ee_p\times \Ee_q\to \Ee_{pq},\;\; \Mm_{p,q}(x,y)= xy\]  induces  an isomorphism
$\Mm_{p,q}:\Ee_p\otimes_A \Ee_q\to \Ee_{pq}$;
\item multiplication in $\Ee$ by elements of $\Ee_e=A$ implements the right and left actions of $A$ on each $\Ee_p$. In particular,  $\Mm_{p,e}$ is an isomorphism.
\end{itemize}
Let $\phi_p:A\to \Ll(\Ee_p)$ be the homomorphism implementing the
left action. The product
system $\Ee$ is said to be {\em essential} if each $\Ee_p$ is an
essential correspondence, i.e. the span of $\phi_p(A)\Ee_p$ is dense in $\Ee_p$ for all $p\in P$. In this case, the map $\Mm_{e,p}$ is also an isomorphism.

If the maps $\phi_p$ take values in $\Kk(\Ee_p)$, then the  product system is called  {\em row-finite} or {\em proper}. If all maps $\phi_p$ are injective, then $\Ee$ is called  {\em faithful}. 
\end{dfn}

\begin{dfn}
Given a product system $\Ee\to P$ over $A$ and a $C^*$-algebra $B$, a map $\psi:\Ee\to B$ is called a {\em Toeplitz representation}
of $\Ee$ if 
\begin{itemize}
\item denoting $\psi_p:=\psi|_{\Ee_p}$, then each  $\psi_p:\Ee_p\to B$ is linear, $\psi_e:A\to B$ is a $*$-homomorphism, and \[\psi_e(\la x,y\ra_p)=\psi_p(x)^*\psi_p(y)\] for all  $x,y\in \Ee_p$; 
\item $\psi_p(x)\psi_q(y)=\psi_{pq}(xy)$ for all $p,q\in P, x\in \Ee_p, y\in \Ee_q$.
\end{itemize}

For each $p\in P$ we write
$\psi^{(p)}$ for the homomorphism $\Kk(\Ee_p)\to B$ obtained by extending the map $\theta_{\xi, \eta}\mapsto \psi_p(\xi)\psi_p(\eta)^*$, where \[\theta_{\xi, \eta}(\zeta)=\xi\la \eta, \zeta\ra.\] 

The
Toeplitz representation $\psi:\Ee\to B$ is {\em Cuntz-Pimsner covariant} if  $\psi^{(p)}(\phi_p(a))=\psi_e(a)$ for all $p\in P$ and all $a\in A$ such that $\phi_p(a)\in \Kk(\Ee_p)$.
\end{dfn}

There is a $C^*$-algebra $\Tt_A(\Ee)$ called the Toeplitz algebra of $\Ee$
and a representation $i_\Ee:\Ee\to \Tt_A(\Ee)$ which is universal in the following sense:
$\Tt_A(\Ee)$ is generated by $i_\Ee(\Ee)$ and
for any representation $\psi :\Ee\to B$ there is a homomorphism $\psi_*:\Tt_A(\Ee)\to B$ such that $\psi_*\circ i_\Ee=\psi$.

There are various extra conditions on a product system $\Ee\to P$  and several other notions of  covariance, which allow to define the Cuntz-Pimsner algebra $\Oo_A(\Ee)$ or the Cuntz-Nica-Pimsner algebra $\Nn\Oo_A(\Ee)$ satisfying certain properties, see \cite{F, SY, CLSV,AM,DK} among others.
We mention that  $\Oo_A(\Ee)$ (or  $\Nn\Oo_A(\Ee)$)
comes with a covariant representation $j_\Ee:\Ee\to \Oo_A(\Ee)$ and is universal in the following sense:
$\Oo_A(\Ee)$ is generated by $j_\Ee(\Ee)$ and
for any covariant representation $\psi :\Ee\to B$ there is a homomorphism $\psi_*:\Oo_A(\Ee)\to B$ such that $\psi_*\circ j_\Ee=\psi$. Under certain conditions, $\Oo_A(\Ee)$ satisfies a gauge invariant uniqueness theorem.

\begin{example}
For  a   product system $\Ee\to P$ with fibers $\Ee_p$  nonzero finitely dimensional Hilbert spaces, in particular $A=\Ee_e=\CC$, let us fix an orthonormal basis $\Bb_p$ in $\Ee_p$. Then a Toeplitz representation $\psi:\Ee\to B$ gives rise to a family of isometries $\{\psi(\xi): \xi\in \Bb_p\}_{p\in P}$ with mutually orthogonal range projections. In this case $\Tt(\Ee)=\Tt_\CC(\Ee)$ is generated by a colection of Cuntz-Toeplitz algebras which interact according to the multiplication maps $\Mm_{p,q}$  in $\Ee$.

A representation $\psi:\Ee\to B$ is  Cuntz-Pimsner covariant if \[ \sum_{\xi\in \Bb_p}\psi(\xi)\psi(\xi)^*=\psi(1)\] for all $p\in P$. 
The Cuntz-Pimsner algebra $\Oo(\Ee)=\Oo_\CC(\Ee)$ is generated by a collection of Cuntz algebras.  N. Fowler proved in \cite{F1} that if the  function $p\mapsto \dim \Ee_p$ is injective, then the algebra ${\mathcal O}(\Ee)$ is simple and purely infinite. For other examples of multidimensional Cuntz algebras, see  \cite{ B}.   
\end{example}
\begin{example}
A row-finite  $k$-graph with no sources $\Lambda$ (see \cite{KP}) determines a product system $\Ee\to \NN^k$ with $\Ee_0=A=C_0(\Lambda^0)$ and $\Ee_n=\overline{C_c(\Lambda^n)}$ for $n\neq 0$ such that we have a $\TT^k$-equivariant isomorphism $\Oo_A(\Ee)\cong C^*(\Lambda)$. Recall that the universal property induces   a gauge action on $\Oo_A(\Ee)$ defined by $\gamma_z(j_\Ee(\xi))=z^nj_\Ee(\xi)$ for $z\in \TT^k$ and $\xi\in \Ee_n$.
\end{example}

The following two definitions and two results are taken from \cite{DHS}, see also \cite{Ka}.

\begin{dfn}
An action $ \beta $ of a locally compact group $ G $ on a product system $ \Ee \to P $ over $A$ is a  family $ (\beta^{p})_{p \in P} $ such that $ \beta^{p} $ is an action of $ G $ on each fiber $\Ee_{p} $ compatible with the action $\alpha=\beta^e$ on $A$, and furthermore, the actions $(\beta^p)_{p\in P}$ are compatible with the multiplication maps $\Mm_{p,q}$ in the sense that 
$$
\beta^{p q}_g(\Mm_{p,q}(x \otimes y)) = \Mm_{p,q}(\beta^{p}_g(x) \otimes \beta^{q}_g(y))
$$
for all $ g \in G $, $ x \in \Ee_{p} $ and $ y \in \Ee_{q} $. 
\end{dfn}

\begin{dfn} \label{cp}
If $ \beta $ is an action of $ G $ on the product system $\Ee \to  P $, we define the crossed product $\Ee \rtimes_{\beta} G $  as the product system indexed by $ P $ with fibers $ \Ee_{p} \rtimes_{\beta^{p}} G $, which are $ C^{\ast} $-correspondences over $ A \rtimes_{\alpha} G $. For $ \zeta \in C_c(G,\Ee_{p}) $ and $ \eta \in C_c(G,\Ee_{q}) $, the product $ \zeta \eta \in C_c(G,\Ee_{p q}) $ is defined  by
$$
(\zeta \eta)(s) = \int_G\Mm_{p,q}(\zeta(t) \otimes \beta^{q}_t(\eta(t^{- 1} s)))dt.
$$
\end{dfn}

\begin{prop}\label{p1}
The set $ \ds \Ee \rtimes_{\beta} G = \bigsqcup_{p \in P} \Ee_{p} \rtimes_{\beta^{p}} G $ with the above multiplication  satisfies all the properties of a product system of $ C^{\ast} $-correspondences over $ A \rtimes_{\alpha} G $.
\end{prop}

\begin{prop}\label{p2}
Suppose that a locally compact group $ G $ acts on a row-finite and faithful product system $ \Ee $ indexed by $ P = (\NN^{k},+) $ via automorphisms $ \beta^{p}_{g} $. Then $ G $ acts on the Cuntz-Pimsner algebra $\Oo_{A}(\Ee) $ via automorphisms denoted by $ \gamma_{g} $. Moreover, if $ G $ is amenable,   then $ \Ee \rtimes_{\beta} G $ is  row-finite and faithful, and
$$
\Oo_{A}(\Ee) \rtimes_{\gamma} G \cong \Oo_{A \rtimes_{\alpha} G}(\Ee \rtimes_{\beta} G).
$$
\end{prop}

Now we define the product system associated to $k$ representations of a compact group $G$. We limit ourselves to finite dimensional unitary representations, even though the definition makes sense in greater generality.
 
\begin{dfn}\label{ps}
Given a compact group $G$ and $k$ finite dimensional unitary representations $\rho_i$ of $G$   on Hilbert spaces $\mathcal H_i$ for $i=1,...,k$, we construct the product system $\Ee=\Ee(\rho_1,...,\rho_k)$ indexed by the commutative monoid $(\mathbb N^k,+)$, with fibers   \[\mathcal E_n=\Hh^n=\mathcal H_1^{\otimes n_1}\otimes\cdots\otimes  \mathcal H_k^{\otimes n_k}\] 
for $n=(n_1,...,n_k)\in \NN^k$, in particular, $A=\mathcal E_0=\mathbb C$. The multiplication maps $\Mm_{n,m}:\Ee_n\times \Ee_m\to \Ee_{n+m}$ in  $\Ee$ are defined using repeatedly the standard isomorphisms $\rho_i\otimes\rho_j\cong \rho_j\otimes \rho_i$ for all $i<j$. The associativity in $\Ee$ follows from the fact that 
\[\Mm_{n+m,p}\circ (\Mm_{n,m}\times id)=\Mm_{n,m+p}\circ ( id\times \Mm_{m,p})\] as maps from $\Ee_n\times \Ee_m\times \Ee_p$ to $\Ee_{n+m+p}.$
 Then $\Ee=\Ee(\rho_1,...,\rho_k)$ is called the product system of the representations $\rho_1,...,\rho_k$.
\end{dfn}
\begin{rmk}
Similarly, a semigroup $P$ of  unitary representations of a group $G$ would determine a product system $\Ee\to P$.
\end{rmk}
\begin{prop}
With notation as in  Definition \ref{ps}, assume $d_i=\dim \Hh_i\ge 2$.  Then the Cuntz-Pimsner algebra $\Oo(\Ee)$ associated to the product system $\Ee\to \NN^k$ described above is isomorphic with the $C^*$-algebra of a rank $k$ graph $\Gamma$ with a single vertex and with $|\Gamma^{\ve_i}|=d_i$. This isomorphism is equivariant for the gauge action. Moreover, \[\Oo(\Ee)\cong \mathcal O_{d_1}\otimes\cdots\otimes \mathcal O_{d_k},\]  where $\mathcal O_n$ is the Cuntz algebra. 
 \end{prop}
 \begin{proof}
Indeed, by choosing a basis in each $\Hh_i$, we get the edges $\Gamma^{\ve_i}$ in a $k$-coloured graph $\Gamma$ with a single vertex. The isomorphisms $\rho_i\otimes\rho_j\cong \rho_j\otimes \rho_i$ determine the factorization rules of the form $ef=fe$ for $e\in \Gamma^{\ve_i}$ and $f\in \Gamma^{\ve_j}$ which obviously satisfy the associativity condition. In particular,  the corresponding isometries in $C^*(\Gamma)$ commute and $\Oo(\Ee)\cong C^*(\Gamma)\cong\mathcal O_{d_1}\otimes\cdots\otimes \mathcal O_{d_k}$, preserving the gauge action. 
\end{proof}

\begin{rmk}
For $d_i\ge 2$, the $C^*$-algebra $\Oo(\Ee)\cong C^*(\Gamma)$ is always simple and purely infinite since it is a  tensor product of simple and purely infinite $C^*$-algebras. If $d_i=1$ for some $i$, then $\Oo(\Ee)$ will contain a copy of $C(\TT)$, so it is not simple. Of course, if $d_i=1$ for all $i$, then $\Oo(\Ee)\cong C(\TT^k)$.
For more on single vertex rank $k$ graphs, see \cite{DY, DY1}. 
\end{rmk}
\begin{prop}
The compact group $G$ acts on each fiber $\Ee_n$ of the  product system $\mathcal E$ via the representation $\rho^n=\rho_1^{\otimes n_1}\otimes\cdots\otimes \rho_k^{\otimes n_k}$. This action is compatible with the multiplication maps and commutes with the gauge action of $\TT^k$.  The crossed product $\mathcal E\rtimes G$ becomes a row-finite and faithful product system indexed by $\mathbb N^k$ over the group $C^*$-algebra $C^*(G)$. Moreover, 
\[\Oo(\Ee) \rtimes G \cong \Oo_{C^*(G)}(\Ee \rtimes G).\]
\end{prop}
\begin{proof}
Indeed, for $g\in G$ and $\xi\in \Ee_n=\Hh^n$ we define $g\cdot\xi=\rho^n(\xi)$ and since $\rho_i\otimes \rho_j\cong \rho_j\otimes \rho_i$, we have $g\cdot(\xi\otimes \eta)=g\cdot\xi\otimes g\cdot \eta$ for $\xi\in \Ee_n, \eta\in\Ee_m$.  Clearly,
\[g\cdot\gamma_z(\xi)=g\cdot(z^n\xi)=z^n(g\cdot\xi)=\gamma_z(g\cdot\xi),\]
so the action of $G$ commutes with the gauge action. Using Proposition \ref{p1}, $\mathcal E\rtimes G$ becomes a product system indexed by $\mathbb N^k$ over $C^*(G)\cong \CC\rtimes G$ with fibers $\Ee_n\rtimes G$. The isomorphism $\Oo(\Ee) \rtimes G \cong \Oo_{C^*(G)}(\Ee \rtimes G)$ follows  from Proposition \ref{p2}.

\end{proof}
\begin{cor}
Since the action of $G$ commutes with the gauge action, the group $G$ acts on the core algebra $\Ff=\Oo(\Ee)^{\TT^k}$.
\end{cor}

\bigskip
\section{The Doplicher-Roberts algebra}

\bigskip

The Doplicher-Roberts algebras $\Oo_\rho$, denoted  by ${\mathcal O}_G$ in \cite{DR1}, were introduced to construct a new duality theory for compact Lie groups $G$ which strengthens the Tannaka-Krein duality. Here $\rho$ is the  $n$-dimensional representation  of $G$ defined by the inclusion $G\subseteq U(n)$ in some unitary group $U(n)$. Let ${\mathcal T}_G$ denote the representation category whose objects are tensor powers $\rho^p=\rho^{\otimes p}$ for $p\ge 0$, and whose arrows are the intertwiners $Hom(\rho^p, \rho^q)$. The group $G$ acts via $\rho$ on the Cuntz algebra ${\mathcal O}_n$ and ${\mathcal O}_G={\mathcal O}_\rho$ is identified in \cite{DR1} with the fixed point algebra ${\mathcal O}_n^G$. If $\sigma$ denotes the restriction to ${\mathcal O}_\rho$ of the canonical endomorphism of $\Oo_n$, then ${\mathcal T}_G$ can be reconstructed from the pair $({\mathcal O}_\rho,\sigma)$. Subsequently, Doplicher-Roberts algebras were associated to any object $\rho$ in a strict tensor $C^*$-category, see \cite {DR2}.

Given  finite dimensional unitary representations $\rho_1, ...,\rho_k$ of a compact group $G$ on Hilbert spaces $\mathcal H_1, ...,  \mathcal H_k$ we will construct a Doplicher-Roberts algebra $\mathcal O_{\rho_1,...,\rho_k}$  from intertwiners 
\[Hom (\rho^n, \rho^m)=\{T\in\Ll({\mathcal H}^n, {\mathcal H}^m)\;\mid \;  T\rho^n(g)=\rho^m(g)T\;\;\forall\;g\in G\},\] 
where for $n=(n_1,...,n_k)\in \mathbb N^k$ the representation $\rho^n=\rho_1^{\otimes n_1}\otimes\cdots\otimes \rho_k^{\otimes n_k}$ acts on $\Hh^n=\mathcal H_1^{\otimes n_1}\otimes \cdots \otimes \mathcal H_k^{\otimes n_k}$. Note that $\rho^0=\iota$ is the trivial representation of $G$,   acting on $\Hh^0=\CC$. This Doplicher-Roberts algebra will be a subalgebra of $\Oo(\Ee)$ for the product system $\Ee$ as in Definition \ref{ps}.

 \begin{lem}\label{cat}
 Consider \[ \Aa_0=\bigcup_{m,n\in \NN^k}\Ll(\mathcal H^n,\mathcal H^m).\] Then the linear span of $\Aa_0$ becomes a $*$-algebra $\Aa$ with appropriate multiplication and involution. This algebra has a natural $\ZZ^k$-grading coming from a gauge action of $\TT^k$. Moreover, the Cuntz-Pimsner algebra
 $\Oo(\Ee)$ of the product system $\Ee=\Ee(\rho_1,...,\rho_k)$ is equivariantly isomorphic to  the $C^*$-closure of $\Aa$ in the unique $C^*$-norm for which the gauge action is isometric.
 \end{lem}

\begin{proof}
Recall  that the Cuntz algebra $\Oo_n$ contains a canonical Hilbert space $\Hh$ of dimension $n$ and it can be constructed as the closure of  the linear span of $\ds \bigcup_{p,q\in \NN}\Ll(\mathcal H^p,\mathcal H^q)$ using embeddings \[\Ll(\Hh^p,\Hh^q)\subseteq \Ll(\Hh^{ p+1},\Hh^{ q+1}),\;\;T\mapsto T\otimes I\] where $\Hh^p=\mathcal H^{\otimes p}$ and $I:\Hh\to\Hh$ is the identity map. This linear span becomes a $*$-algebra with a multiplication given by composition and an involution (see \cite{DR1} and  Proposition 2.5 in \cite{KPW}). 

Similarly, for all $r\in \NN^k$, we consider embeddings $\Ll(\Hh^n,\Hh^m)\subseteq \Ll(\Hh^{n+r},\Hh^{m+r})$ given by $T\mapsto T\otimes I_r$, where $I_r:{\mathcal H}^r\to {\mathcal H}^r$ is the identity map,  
 and endow  $\Aa$ with a multiplication given by composition and an involution. More precisely, if $S\in \Ll(\Hh^n,\Hh^m)$ and $T\in \Ll(\Hh^q,\Hh^p)$, then the product $ST$ is
 \[(S\otimes I_{p\vee n-n})\circ (T\otimes I_{p\vee n-p})\in \Ll(\Hh^{q+p\vee n-p},\Hh^{m+p\vee n-n}),\]
where  we write $p\vee n$ for the coordinatewise maximum. This multiplication is well defined in $\Aa$ and is associative.
The adjoint of $T\in \Ll(\Hh^n,\Hh^m)$ is $T^*\in \Ll(\Hh^m,\Hh^n)$.

There is a natural $\mathbb Z^k$-grading on $\Aa$ given by the gauge action $\gamma$ of $\TT^k$, where for $z=(z_1,...,z_k)\in \TT^k$ and $T\in \Ll(\Hh^n,\Hh^m)$ we define \[\gamma_z(T)(\xi)=z_1^{m_1-n_1}\cdots z_k^{m_k-n_k}T(\xi).\]
Adapting the argument  in Theorem 4.2 in \cite{DR2} for $\ZZ^k$-graded $C^*$-algebras, the $C^*$-closure of $\Aa$ in the unique $C^*$-norm for which $\gamma_z$ is isometric is well defined.  The map
\[(T_1,...,T_k)\mapsto  T_1\otimes\cdots\otimes T_k, \]
where 
 \[T_1\otimes\cdots\otimes T_k: \mathcal H^n\to \mathcal H^m,\; (T_1\otimes\cdots\otimes T_k)(\xi_1\otimes\cdots\otimes \xi_k)=T_1(\xi_1)\otimes\cdots\otimes T_k(\xi_k)\] for  $T_i\in  \Ll(\mathcal H_i^{n_i},\mathcal H_i^{m_i})$ for $i=1,...,k$ preserves the gauge action and it can be extended to an equivariant isomorphism from $\Oo(\Ee)\cong \Oo_{d_1}\otimes\cdots\otimes\Oo_{d_k}$ to the $C^*$-closure of $\Aa$. Note that the closure  of $\ds \bigcup_{n\in \NN^k}\Ll(\mathcal H^n,\mathcal H^n)$ is isomorphic to the core $\Ff=\Oo(\Ee)^{\TT^k}$, the fixed point algebra under the gauge action, which is a UHF-algebra.
\end{proof}

To define the Doplicher-Roberts algebra $\mathcal O_{\rho_1,...,\rho_k}$, we will again identify $Hom(\rho^n,\rho^m)$ with a subset of $Hom(\rho^{n+r},\rho^{m+r})$ for each $r\in \mathbb N^k$, via $T\mapsto T\otimes I_r$. 
After this identification, it follows that the linear span ${}^0{\mathcal O}_{\rho_1, ..., \rho_k}$ of $\ds \bigcup_{m,n\in\NN^k}Hom(\rho^n, \rho^m)\subseteq \Aa_0$ has a natural multiplication and involution inherited from $\Aa$. Indeed, a computation shows that if $S\in Hom(\rho^n, \rho^m)$ and $T\in Hom(\rho^q,\rho^p)$, then $S^*\in Hom(\rho^m, \rho^n)$ and 
\[(S\otimes I_{p\vee n-n})\circ (T\otimes I_{p\vee n-p})\rho^{q+p\vee n-p}(g)=\]\[=\rho^{m+p\vee n-n}(g)(S\otimes I_{p\vee n-n})\circ (T\otimes I_{p\vee n-p}),\]
so $(S\otimes I_{p\vee n-n})\circ (T\otimes I_{p\vee n-p})\in Hom(\rho^{q+p\vee n-p}, \rho^{m+p\vee n-n})$ and ${}^0{\mathcal O}_{\rho_1, ..., \rho_k}$ is closed under these operations.  Since the action of $G$ commutes with the gauge action, there is a natural $\mathbb Z^k$-grading of ${}^0{\mathcal O}_{\rho_1,...,\rho_k}$ given by the gauge action $\gamma$ of $\TT^k$ on  $\Aa$.

It follows that the closure ${\mathcal O}_{\rho_1,..., \rho_k}$ of ${}^0{\mathcal O}_{\rho_1, ...,\rho_k}$ in $\Oo(\Ee)$  is well defined, obtaining the Doplicher-Roberts algebra  associated to the 
representations $\rho_1, ...,\rho_k$. This $C^*$-algebra also has a $\mathbb Z^k$-grading and a gauge action of $\TT^k$. By construction, ${\mathcal O}_{\rho_1,..., \rho_k}\subseteq \Oo(\Ee)$.
\begin{rmk}
For a compact Lie group $G$, our Doplicher-Roberts algebra ${\mathcal O}_{\rho_1,..., \rho_k}$ is Morita equivalent with the higher rank Doplicher-Roberts algebra $\Dd$ in \cite{AM}. It is also the section $C^*$-algebra of a Fell bundle over $\ZZ^k$.
\end{rmk}
\begin{thm}
Let $\rho_i$ be  finite dimensional unitary representations  of a compact group $G$ on Hilbert spaces $\mathcal H_i$ of dimensions $d_i\ge 2$ for $i=1,...,k$. Then the Doplicher-Roberts algebra ${\mathcal O}_{\rho_1,...,\rho_k}$    is isomorphic to the fixed point algebra ${\mathcal O}(\Ee)^G\cong (\mathcal O_{d_1}\otimes\cdots\otimes \mathcal O_{d_k})^G$, where $\Ee=\Ee(\rho_1,...,\rho_k)$ is the product system described in Definition \ref{ps}.
\end{thm}
\begin{proof}
We  known from Lemma \ref{cat} that  ${\mathcal O}(\Ee)$ is isomorphic to the $C^*$-algebra generated by the linear span of 
$\ds \Aa_0= \bigcup_{m,n\in\NN^k}\Ll({\mathcal H}^n, {\mathcal H}^m)$.
The group $G$ acts on $\Ll({\mathcal H}^n, {\mathcal H}^m)$  by \[(g\cdot T)(\xi)=\rho^m(g)T(\rho^n(g^{-1})\xi)\] and the fixed point set is $Hom(\rho^n, \rho^m)$. Indeed,  we have $g\cdot T=T$ if and only if
$T\rho^n(g)=\rho^m(g)T$. This action is compatible with the embeddings and the operations, so it extends to the $*$-algebra $\Aa$ and the fixed point algebra is the linear span of $\ds \bigcup_{m,n\in\NN^k}Hom(\rho^n, \rho^m)$.

It follows that ${}^0{\mathcal O}_{\rho_1,...,\rho_k}\subseteq {\mathcal O}(\Ee)^G$ and therefore its closure ${\mathcal O}_{\rho_1,...,\rho_k}$ is isomorphic to a subalgebra of ${\mathcal O}(\Ee)^G$. For the other inclusion, any element in ${\mathcal O}(\Ee)^G$ can be approximated with
an element from  ${}^0{\mathcal O}_{\rho_1,...,\rho_k}$, hence ${\mathcal O}_{\rho_1,...,\rho_k}=\Oo(\Ee)^G$.
\end{proof}
\begin{rmk}
By left tensoring with $I_r$ for $r\in \NN^k$, we obtain  some canonical unital endomorphisms $\sigma_r$ of ${\mathcal O}_{\rho_1, ...,\rho_k}$.
\end{rmk}

In the next section, we will show that in many cases, $\mathcal O_{\rho_1,...,\rho_k}$ is isomorphic  to a corner of $C^*(\Lambda)$ for a rank $k$ graph $\Lambda$, so in some cases we can compute its $K$-theory. It would be nice to express the $K$-theory of $\Oo_{\rho_1,...,\rho_k}$ in terms of the endomorphisms $\pi\mapsto \pi\otimes \rho_i$ of the representation ring $\Rr(G)$.


\bigskip


\section{The rank $k$ graphs}

\bigskip

For convenience, we first collect some facts about  higher rank graphs, introduced in \cite{KP}. A  rank $k$ graph  or $k$-graph $(\Lambda, d)$
consists of a countable small category $\Lambda$ with range and 
source maps $r$ and $s$  together with a functor 
$d : \Lambda \rightarrow \NN^k$ called the degree map, satisfying the  factorization 
property: for every $\lambda \in \Lambda$ and all $m, n \in \NN^k$ with $d( \lambda ) = m + n$, there are unique elements
$\mu , \nu \in \Lambda$ such that $\lambda = \mu \nu $ and
$d( \mu ) = m$, $d( \nu ) = n$. For $n \in \NN^k$ we write
$\Lambda^n := d^{-1} (n)$ and call it  the set of paths of degree $n$. The elements in $\Lambda^{\ve_i}$ are called edges and the elements in $\Lambda^0$ are called vertices.

A $k$-graph $\Lambda$ can be constructed from $\Lambda^0$ and from its $k$-coloured skeleton $\Lambda^{\ve_1}\cup\cdots \cup\Lambda^{\ve_k}$ using a complete and associative collection of commuting squares or factorization rules, see \cite{S}.

The $k$-graph $\Lambda$ is {\em row-finite} if for all $n\in \NN^k$ and all $v\in \Lambda^0$ the set $v\Lambda^n := \{\lambda\in\Lambda^n : r(\lambda) = v\}$ is finite. It has no sources if $v\Lambda^n\neq \emptyset$ for all $v\in \Lambda^0$ and $n\in\NN^k$.
A $k$-graph  $\Lambda$ is said to be {\em irreducible} (or {\em strongly connected}) if, for every $u,v\in \Lambda^0$, there is $\lambda\in \Lambda$    such that $u = r(\lambda)$ and $v = s(\lambda)$.

Recall that $C^*(\Lambda)$ is the universal $C^*$-algebra generated by a family $\{S_\lambda: \lambda\in \Lambda \}$ of partial isometries satisfying:
\begin{itemize}
\item  $\{S_v:v\in \Lambda^0\}$ is a family of mutually orthogonal projections, 

\item $S_{\lambda\mu}=S_\lambda S_\mu$  for all $\lambda, \mu\in \Lambda$  such that $s(\lambda) = r(\mu)$,

\item $S_\lambda^*S_\lambda = S_{s(\lambda)}$ for all $\lambda\in \Lambda$,

\item for all  $v\in \Lambda^0$ and $n\in \NN^k$ we have \[ S_v=\sum_{\lambda\in v\Lambda^n}S_\lambda S_\lambda^*.\] 
\end{itemize}

A $k$-graph $\Lambda$ is said to satisfy the {\em aperiodicity condition}  if for every vertex $v\in \Lambda^0$ there is an infinite path $x\in v\Lambda^\infty$ such that $\sigma^mx\neq \sigma^nx$ for all $m\neq n$ in $\NN^k$, where $\sigma^m:\Lambda^\infty\to \Lambda^\infty$ are the shift maps.
We say that $\Lambda$ is {\em cofinal} if for every $x\in \Lambda^\infty$ and $v\in \Lambda^0$ there is $\lambda\in \Lambda$ and $n\in \NN^k$  such that $s(\lambda)=x(n)$ and $r(\lambda)=v$.

Assume that  $\Lambda$ is row finite with no sources and that it satisfies  the aperiodicity condition. Then $C^*(\Lambda)$ is simple if and only if $\Lambda$ is cofinal (see Proposition 4.8 in \cite{KP} and  Theorem 3.4 in \cite{RS}). 

We say that a path $\mu\in \Lambda$ is a loop with an entrance if $s(\mu)=r(\mu)$ and there exists $\alpha\in s(\mu)\Lambda$ such that $d(\mu)\ge d(\alpha)$ and there is no $\beta\in \Lambda$ with $\mu= \alpha\beta$. We say that every vertex {\em connects to a loop with an entrance} if for every $v\in \Lambda^0$  there are a loop with an entrance $\mu\in \Lambda$  and a path $\lambda\in \Lambda$ with    $r(\lambda)=v$ and   $s(\lambda)=r(\mu)=s(\mu)$. If $\Lambda$ satisfies  the aperiodicity condition and  every vertex connects to a loop with an entrance, then $C^*(\Lambda)$ is purely infinite (see Proposition 4.9 in \cite{KP} and Proposition 8.8 in \cite{ S06}).


Given finitely dimensional unitary representations $\rho_i$  of a compact group $G$ on Hilbert spaces $\mathcal H_i$ for $i=1,...,k$, we want to  construct a  rank $k$ graph $\Lambda=\Lambda(\rho_1,...,\rho_k)$. Let $R$ be the set of  equivalence classes of irreducible summands $\pi:G\to U(\Hh_\pi)$ which appear in the tensor powers $\rho^n=\rho_1^{\otimes n_1}\otimes\cdots \otimes \rho_k^{\otimes n_k}$ for $n\in \NN^k$ as in \cite{MRS}. 
Take $\Lambda^0=R$   and for each $i=1,...,k$ consider the set of  edges $\Lambda^{\ve_i}$ which are uniquely determined by the  matrices $M_i$ with entries  
\[M_i(w,v)=|\{e\in \Lambda^{\ve_i}: s(e)=v, r(e)=w\}|=\dim Hom(v,w\otimes \rho_i),\] 
where $v,w\in R$.
The matrices $M_i$ commute since $\rho_i\otimes\rho_j\cong \rho_j\otimes \rho_i$ and therefore \[\dim Hom(v,w\otimes \rho_i\otimes\rho_j)=\dim Hom(v,w\otimes \rho_j\otimes\rho_i)\] 
for all $i<j$. 
 This will allow us to fix some bijections  \[\lambda_{ij}:\Lambda^{\ve_i}\times_{\Lambda^0}\Lambda^{\ve_j}\to \Lambda^{\ve_j}\times_{\Lambda^0}\Lambda^{\ve_i}\] for all $1\le i<j\le k$, which will determine the commuting squares of $\Lambda$. As usual,
\[\Lambda^{\ve_i}\times_{\Lambda^0}\Lambda^{\ve_j}=\{(e,f)\in \Lambda^{\ve_i}\times \Lambda^{\ve_j}: s(e)=r(f)\}.\]

For $k\ge 3$ we also need to verify that $\lambda_{ij}$ can be chosen to satisfy the associativity condition, i.e.
 \[(id_\ell\times \lambda_{ij})(\lambda_{i\ell}\times id_j)(id_i\times \lambda_{j\ell})=(\lambda_{j\ell}\times id_i)(id_j\times \lambda_{i\ell})(\lambda_{ij}\times id_\ell)\] 
 as bijections from $\Lambda^{\ve_i}\times_{\Lambda^0}\Lambda^{\ve_j}\times_{\Lambda^0}\Lambda^{\ve_\ell}$ to $\Lambda^{\ve_\ell}\times_{\Lambda^0}\Lambda^{\ve_j}\times_{\Lambda^0}\Lambda^{\ve_i}$ for all $i<j<\ell$. 

\begin{rmk}
Many times $R=\hat{G}$, so $\Lambda^0=\hat{G}$, for example if $\rho_i$ are faithful and $\rho_i(G)\subseteq SU(\Hh_i)$  or if $G$ is finite, $\rho_i$ are faithful and $\dim \rho_i\ge2$  for all $i=1,...,k$, see Lemma 7.2 and Remark 7.4 in \cite{KPRR}.

\end{rmk} 

\begin{prop}
Given representations $\rho_1,...,\rho_k$ as above, assume that $\rho_i$ are faithful and that $R=\hat{G}$. Then each choice of bijections $\lambda_{ij}$ satisfying the associativity condition determines a rank $k$ graph $\Lambda$   which is cofinal and locally finite with no sources.
\end{prop}

\begin{proof}
Indeed, the sets $\Lambda^{\ve_i}$ are uniquely determined and the choice of bijections $\lambda_{ij}$ satisfying the associativity condition will be enough to determine $\Lambda$. Since the entries of the matrices $M_i$ are finite and there are no zero rows, the graph is locally finite with no sources. To prove that $\Lambda$ is cofinal, fix a vertex $v\in \Lambda^0$ and an infinite path $x\in \Lambda^\infty$. Arguing as in Lemma 7.2 in \cite{KPRR}, any $w\in \Lambda^0$, in particular $w=x(n)$ for a fixed $n$ can be joined by a path  to $v$, so there is $\lambda\in \Lambda$ with $s(\lambda)=x(n)$ and $r(\lambda)=v$. See also Lemma 3.1 in \cite{MRS}.
\end{proof}
\begin{rmk}
Note that the entry $M_i(w,v)$ is just the multiplicity of the irreducible representation $v$ in $w\otimes \rho_i$ for $i=1,...,k$. If $\rho_i^*=\rho_i$, the matrices $M_i$ are  symmetric since \[\dim Hom(v, w\otimes \rho_i)=\dim Hom(\rho^*_i\otimes v,w).\] Here $\rho^*_i$ denotes the dual representation, defined by $\rho_i^*(g)=\rho_i(g^{-1})^t$, and equal in our case to the conjugate representation $\bar{\rho_i}$.  

For $G$ finite, these matrices are finite, and the entries $M_i(w,v)$ can be computed using the character table of $G$. For $G$ infinite,  the Clebsch-Gordan relations can be used  to determine the numbers $M_i(w,v)$. Since the bijections $\lambda_{ij}$ in general are not unique,  the rank $k$ graph $\Lambda$ is not unique, as illustrated in some examples.  It is an open question how the $C^*$-algebra $C^*(\Lambda)$  depends in general on the factorization rules.
\end{rmk}

To relate  the Doplicher-Roberts algebra $\Oo_{\rho_1,...,\rho_k}$ to a rank $k$ graph $\Lambda$, we mimic the construction in \cite{MRS}.
 For each edge $e\in \Lambda^{\ve_i}$, choose an isometric intertwiner \[T_e:
\Hh_{s(e)}\to \Hh_{r(e)}\otimes \Hh_i\] in such a way that 
\[\Hh_\pi\otimes \Hh_i=\bigoplus_{e\in \pi\Lambda^{\ve_i}}T_eT_e^*(\Hh_\pi\otimes \Hh_i)\]
for all $\pi\in \Lambda^0$, i.e. the edges in $\Lambda^{\ve_i}$ ending at $\pi$ give a specific decomposition of $\Hh_\pi\otimes \Hh_i$ into irreducibles. When $\dim Hom(s(e), r(e)\otimes \rho_i)\ge 2$ we must choose a basis of isometric intertwiners with orthogonal ranges, so in general $T_e$ is not unique. In fact, specific choices for the isometric intertwiners $T_e$ will determine the factorization rules in $\Lambda$ and whether they satisfy  the associativity condition or not. 

Given $e\in \Lambda^{\ve_i}$ and $f\in \Lambda^{\ve_j}$ with $r(f)=s(e)$,  we know how to multiply $T_e\in Hom(s(e),r(e)\otimes \rho_i)$ with $T_f\in Hom(s(f),r(f)\otimes \rho_j)$ in the algebra $\Oo_{\rho_1,...,\rho_k}$, by viewing $Hom(s(e),r(e)\otimes \rho_i)$ as a subspace of $Hom(\rho^n,\rho^m)$ for some $m,n$ and similarly for $Hom(s(f),r(f)\otimes \rho_j)$. We choose edges $e'\in \Lambda^{\ve_i}, f'\in \Lambda^{\ve_j}$ with $s(f)=s(e'), r(e)=r(f'), r(e')=s(f')$ such that $T_eT_f=T_{f'}T_{e'}$, where $T_{f'}\in Hom(s(f'),r(f')\otimes \rho_j)$ and $T_{e'}\in Hom(s(e'),r(e')\otimes \rho_i)$.
This is possible since 
\[T_eT_f=(T_e\otimes I_j)\circ T_f\in Hom(s(f),r(e)\otimes \rho_i\otimes \rho_j),\]
\[T_{f'}T_{e'}=(T_{f'}\otimes I_i)\circ T_{e'}\in Hom(s(e'),r(f')\otimes \rho_j\otimes \rho_i),\]
 and $\rho_i\otimes \rho_j\cong \rho_j\otimes \rho_i$. In this case we declare that $ef=f'e'$. Repeating this process, we obtain bijections $\lambda_{ij}:\Lambda^{\ve_i}\times_{\Lambda^0}\Lambda^{\ve_j}\to \Lambda^{\ve_j}\times_{\Lambda^0}\Lambda^{\ve_i}$.  Assuming that the associativity conditions are satisfied, we obtain a $k$-graph $\Lambda$. 
 
 We write $T_{ef}=T_eT_f=T_{f'}T_{e'}=T_{f'e'}$.
A finite path $\lambda\in \Lambda^n$ is a concatenation of edges and determines by composition a unique intertwiner \[T_\lambda:\Hh_{s(\lambda)}\to \Hh_{r(\lambda)}\otimes\Hh^n.\] Moreover, the paths $\lambda\in \Lambda^n$ with $r(\lambda)=\iota$, the trivial representation, provide an explicit decomposition of $\Hh^n=\Hh_1^{\otimes n_1}\otimes\cdots\otimes\Hh_k^{\otimes n_k}$ into irreducibles, hence
\[\Hh^n=\bigoplus_{\lambda\in\iota \Lambda^n}T_\lambda T_\lambda^*(\Hh^n).\]

\begin{prop} Assuming that the choices of isometric intertwiners $T_e$ as above determine a $k$-graph $\Lambda$, then
the family 
\[\{T_\lambda T^*_\mu: \lambda\in\Lambda^m, \mu\in\Lambda^n, r(\lambda)=r(\mu)=\iota, s(\lambda)=s(\mu)\}\]
is a basis for $Hom(\rho^n, \rho^m)$ and each $T_\lambda T^*_\mu$ is a partial isometry.
\end{prop}
\begin{proof}
Each pair of paths $\lambda, \mu$ with $d(\lambda)=m, d(\mu)=n$ and $r(\lambda)=r(\mu)=\iota$ determines a pair of irreducible summands $T_\lambda(\Hh_{s(\lambda)}), T_\mu(\Hh_{s(\mu)})$ of $\Hh^m$ and $ \Hh^n$ respectively. By Schur's lemma, the space of intertwiners of these representations is trivial unless $s(\lambda)=s(\mu)$ in which case it is the one dimensional space spanned by $T_\lambda T_\mu^*$. It follows that any element of $Hom(\rho^n, \rho^m)$ can be uniquely represented as a linear combination of elements $T_\lambda T_\mu^*$ where $s(\lambda)=s(\mu)$. Since $T_\mu$ is isometric, $T_\mu^*$ is a partial isometry with range $\Hh_{s(\mu)}$ and hence $T_\lambda T_\mu^*$ is also a partial isometry whenever $s(\lambda)=s(\mu)$.
\end{proof}

\begin{thm}\label{t1}
Consider $\rho_1,..., \rho_k$  finite dimensional unitary representations of a compact group $G$ and let $\Lambda$ be the $k$-coloured graph with $\Lambda^0=R\subseteq \hat{G}$ and edges $\Lambda^{\ve_i}$ determined by the incidence matrices $M_i$ defined above.  Assume that the factorization rules determined by the choices of $T_e\in Hom(s(e),r(e)\otimes \rho_i)$ for all edges $e\in \Lambda^{\ve_i}$ satisfy  the associativity condition, so $\Lambda$ becomes a rank $k$ graph. If we consider $P\in C^*(\Lambda)$,
\[P=\sum_{\lambda\in\iota \Lambda^{(1,...,1)}}S_\lambda S_\lambda^*,\]
where $\iota$ is the trivial representation, then there is a $*$-isomorphism of the Doplicher-Roberts algebra $\Oo_{\rho_1,...,\rho_k}$ onto the corner $PC^*(\Lambda)P$.
\end{thm}
\begin{proof}

Since $C^*(\Lambda)$ is generated by linear combinations of $S_\lambda S_\mu^*$ with $s(\lambda)=s(\mu)$ (see Lemma 3.1 in \cite{KP}), we first define the maps \[\phi_{n,m}:Hom(\rho^n, \rho^m)\to C^*(\Lambda),\;\; \phi_{n,m}(T_\lambda T_\mu^*)=S_\lambda S_\mu^*\] where $s(\lambda)=s(\mu)$ and $r(\lambda)=r(\mu)=\iota$. Since 
$S_\lambda S_\mu^*=PS_\lambda S_\mu^*P$, the maps $\phi_{n,m}$ take values in $PC^*(\Lambda)P$. We claim that for any $r\in \NN^k$ we have
\[\phi_{n+r,m+r}(T_\lambda T_\mu^*\otimes I_r)=\phi_{n,m}(T_\lambda T_\mu^*).\]
This is because 
\[\Hh_{s(\lambda)}\otimes \Hh^r=\bigoplus_{\nu\in s(\lambda)\Lambda^r}T_\nu T_\nu^*(\Hh_{s(\lambda)}\otimes \Hh^r),\]
 so that
\[T_\lambda T_\mu^*\otimes I_r=\sum_{\nu\in s(\lambda)\Lambda^r}(T_\lambda\otimes I_r)(T_\nu T_\nu^*)(T_\mu^*\otimes I_r)=\sum_ {\nu\in s(\lambda)\Lambda^r} T_{\lambda\nu}T_{\mu\nu}^*\]
and\[S_\lambda S_\mu^*=\sum_{\nu\in s(\lambda)\Lambda^r}S_\lambda(S_\nu S_\nu^*)S_\mu^*=\sum_{\nu\in s(\lambda)\Lambda^r}S_{\lambda\nu}S_{\mu\nu}^*.\]
The maps $\phi_{n,m}$ determine a map $\phi: {}^0\Oo_{\rho_1,...,\rho_k}\to PC^*(\Lambda)P$ which is linear, $*$-preserving and multiplicative. 
Indeed, 
\[\phi_{n,m}(T_\lambda T_\mu^*)^*=(S_\lambda S_\mu^*)^*=S_\mu S_\lambda^*=\phi_{m,n}(T_\mu T_\lambda^*).\]
Consider now $T_\lambda T_\mu^*\in Hom(\rho^n, \rho^m),\;\; T_\nu T_\omega^*\in Hom (\rho^q,\rho^p)$
with $s(\lambda)=s(\mu), s(\nu)=s(\omega), r(\lambda)=r(\mu)=r(\nu)=r(\omega)=\iota$. Since for all $n\in \NN^k$
\[\sum_{\lambda\in\iota\Lambda^n}T_\lambda T_\lambda^*=I_n,\]
we get
\[T_\mu^*T_\nu= \begin{cases} T_\beta^*\;\; \text{if}\;\; \mu=\nu\beta\\ T_\alpha \;\;\text{if}\;\; \nu=\mu\alpha\\0 \;\;\text{otherwise,}\end{cases}\]
hence
\[\phi((T_\lambda T_\mu^*)(T_\nu T_\omega^*))=\begin{cases} \phi(T_\lambda T_{\omega\beta}^*)=S_\lambda S_{\omega\beta}^*\;\;\text{if} \;\; \mu=\nu\beta\\
\phi(T_{\lambda\alpha}T_\omega^*)=S_{\lambda\alpha}S_\omega^*\;\;\text{if}\;\; \nu=\mu\alpha\\0 \;\;\text{otherwise.}\end{cases}\]
On the other hand, from  Lemma 3.1 in \cite{KP},
\[S_\lambda S_\mu^* S_\nu S_{\omega}^*=\begin{cases} S_\lambda S_{\omega\beta}^*\;\;\text{if} \;\; \mu=\nu\beta\\
S_{\lambda\alpha}S_\omega^*\;\;\text{if}\;\; \nu=\mu\alpha\\0 \;\;\text{otherwise,}\end{cases}\]
hence \[\phi((T_\lambda T_\mu^*)(T_\nu T_\omega^*))=\phi(T_\lambda T_\mu^*)\phi(T_\nu T_\omega^*).\]

Since $PS_\lambda S_\mu^*P=\phi_{n,m}(T_\lambda T_\mu^*)$ if $r(\lambda)=r(\mu)=\iota$ and $s(\lambda)=s(\mu)$, it follows that $\phi$ is surjective. Injectivity follows from the fact that $\phi$ is equivariant for the gauge action.
\end{proof}

\begin{cor}
If the $k$-graph $\Lambda$ associated to $\rho_1, ...,\rho_k$ is cofinal, it satisfies the aperiodicity condition and every vertex connects to a loop with an entrance, then the Doplicher-Roberts algebra $\Oo_{\rho_1,...,\rho_k}$ is simple and purely infinite,  and is Morita equivalent with $C^*(\Lambda)$. 
\end{cor}

\begin{proof}
This follows from the fact that $C^*(\Lambda)$ is simple and purely infinite and because $PC^*(\Lambda)P$ is a full corner. 
\end{proof}

\begin{rmk}
There is a groupoid $\Gg_\Lambda$ associated to a row-finite rank $k$ graph $\Lambda$ with no sources, see \cite{KP}. By taking the pointed groupoid $\Gg_\Lambda(\iota)$, the reduction to the set of infinite paths with range $\iota$, under the same conditions as in Theorem \ref{t1}, we get an isomorphism of the Doplicher-Roberts algebra $\Oo_{\rho_1,...,\rho_k}$ onto $C^*(\Gg_\Lambda(\iota))$.
\end{rmk}

\bigskip

\section{Examples}
\bigskip

\begin{example}
Let $G=S_3$ be the symmetric group with $\hat{G}=\{\iota, \epsilon, \sigma\}$ and character table

\bigskip

\begin{center}
\begin{tabular}{c|c|r|r}{}&$(1)$ &$(12)$& $(123)$\\
\hline
$\iota$& $1$&$1$&$1$\\
\hline
$\epsilon$ &$1$&$-1$&$1$\\
\hline
$\sigma$ &$2$&$0$&$-1$\\
\hline

\end{tabular}
\end{center}

\bigskip
Here $\iota$ denotes the trivial representation, $\epsilon$ is the sign representation and $\sigma$ is  an irreducible $2$-dimensional representation, for example
\[\sigma((12))=\left[\begin{array}{rr}-1&-1\\0&1\end{array}\right],\;\;\sigma((123))=\left[\begin{array}{rr}-1&-1\\1&0\end{array}\right].\]

By choosing  $\rho_1=\sigma$ on $\Hh_1=\CC^2$ and $\rho_2=\iota+\sigma$ on $\Hh_2=\CC^3$, we get a product system $\Ee\to \NN^2$ and  an action of $S_3$ on $\Oo(\Ee)\cong \mathcal O_2\otimes \mathcal O_3$ with fixed point algebra $\Oo(\Ee)^{S_3}\cong \Oo_{\rho_1,\rho_2}$ isomorphic to a corner of the $C^*$-algebra of a  rank $2$ graph $\Lambda$.  The set of vertices is $\Lambda^0=\{\iota,\epsilon, \sigma\}$  and the edges are given by the incidence matrices

\[ M_1=\left[\begin{array}{ccc}0&0&1\\0&0&1\\1&1&1\end{array}\right], \;\; M_2=\left[\begin{array}{ccc}1&0&1\\0&1&1\\1&1&2\end{array}\right].\]
This is because
\[\iota\otimes\rho_1=\sigma,\; \epsilon\otimes \rho_1=\sigma,\; \sigma\otimes \rho_1=\iota+\epsilon+\sigma,\]
\[\iota\otimes\rho_2=\iota+\sigma,\; \epsilon\otimes\rho_2=\epsilon+\sigma,\; \sigma\otimes\rho_2=\iota+\epsilon+2\sigma.\]

We label the blue edges by $e_1, ..., e_5$ and the red edges by $f_1,...,f_8$ as in the figure

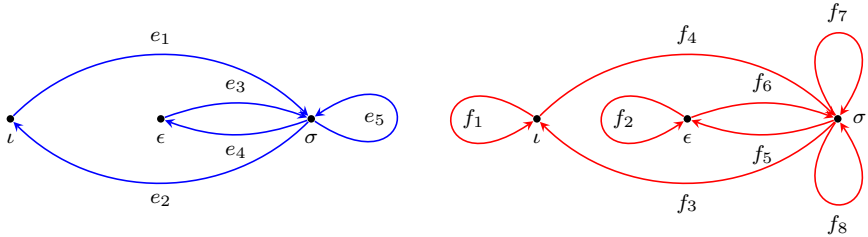
\[
\begin{tikzpicture}
\renewcommand{\ss}{\scriptstyle}
\node[inner sep=1.0pt, circle, fill=black]  (u) at (-6,0) {};
\node[below] at (u.south)  {$\ss \iota$};
\node[inner sep=1.0pt, circle, fill=black]  (v) at (-4,0) {};
\node[below] at (v.south)  {$\ss \epsilon$};
\node[inner sep=1.0pt, circle, fill=black]  (w) at (-2,0) {};
\node[below] at (w.south)  {$\ss \sigma$};
\draw[-stealth, semithick, blue] (u) to [out=45, in=135] node[above,black] {$\ss e_1$}  (w);
\draw[-stealth, semithick, blue] (w) to [out=-135, in=-45] node[below,black] {$\ss e_2$}  (u);

\draw[-stealth, semithick, blue] (v) to [out=20, in=160] node[above,black] {$\ss e_3$}  (w);
\draw[-stealth, semithick, blue] (w) to [out=-160, in=-20]  node[below,black] {$\ss e_4$} (v);
\draw[-stealth, semithick, blue] (w) .. controls (-.5,-1) and (-.5, 1) ..   node[left,black] {$\ss e_5$} (w);
\node[inner sep=1.0pt, circle, fill=black]  (x) at (1,0) {};
\node[inner sep=1.0pt, circle, fill=black]  (y) at (3,0) {};
\node[inner sep=1.0pt, circle, fill=black]  (z) at (5,0) {};
\draw[-stealth, semithick, red] (x) to [out=45, in=135] node[above,black] {$\ss f_4$}  (z);
\draw[-stealth, semithick, red] (z) to [out=-135, in=-45] node[below,black] {$\ss f_3$}  (x);
\draw[-stealth, semithick, red] (y) to [out=20, in=160] node[above,black] {$\ss f_6$}  (z);
\draw[-stealth, semithick, red] (z) to [out=-160, in=-20] node[below,black] {$\ss f_5$}  (y);
\draw[-stealth, semithick, red] (x) .. controls (-0.5,1) and (-0.5, -1) ..   node[right,black] {$\ss f_1$} (x);
\draw[-stealth, semithick, red] (z) .. controls (4,1.5) and (6, 1.5) ..   node[above,black] {$\ss f_7$} (z);
\draw[-stealth, semithick, red] (y) .. controls (1.5,1) and (1.5,-1) ..   node[right,black] {$\ss f_2$} (y);
\draw[-stealth, semithick, red] (z) .. controls (4,-1.5) and (6, -1.5) ..   node[below,black] {$\ss f_8$} (z);
\node[below] at (x.south)  {$\ss \iota$};
\node[below] at (y.south)  {$\ss \epsilon$};
\node[right] at (z.east)  {$\ss \sigma$};
\end{tikzpicture}
\]

The isometric intertwiners are
\[T_{e_1}:\Hh_\iota\to \Hh_\sigma\otimes \Hh_1, \; T_{e_2}:\Hh_\sigma\to \Hh_\iota\otimes \Hh_1, \;T_{e_3}:\Hh_\epsilon\to \Hh_\sigma\otimes \Hh_1,\]
\[T_{e_4}:\Hh_\sigma\to \Hh_\epsilon\otimes \Hh_1,\; T_{e_5}:\Hh_\sigma\to\Hh_\sigma\otimes\Hh_1,\]
\[T_{f_1}:\Hh_\iota\to \Hh_\iota\otimes\Hh_2,\; T_{f_2}:\Hh_\epsilon\to \Hh_\epsilon\otimes\Hh_2,\; T_{f_3}:\Hh_\sigma\to \Hh_\iota\otimes\Hh_2,\]
\[T_{f_4}:\Hh_\iota\to \Hh_\sigma\otimes\Hh_2,\; T_{f_5}:\Hh_\sigma\to\Hh_\epsilon\otimes \Hh_2,\; T_{f_6}:\Hh_\epsilon\to\Hh_\sigma\otimes \Hh_2,\]
\[ T_{f_7}, T_{f_8}:\Hh_\sigma\to \Hh_\sigma\otimes\Hh_2\]
such that
\[T_{e_1}T_{e_1}^*+T_{e_3}T_{e_3}^*+T_{e_5}T_{e_5}^*=I_\sigma\otimes I_1, \; T_{e_2}T_{e_2}^*=I_\iota\otimes I_1,\; T_{e_4}T_{e_4}^*=I_\epsilon\otimes I_1,\]
\[T_{f_1}T_{f_1}^*+T_{f_3}T_{f_3}^*=I_\iota\otimes I_2,\; T_{f_2}T_{f_2}^*+T_{f_5}T_{f_5}^*=I_\epsilon\otimes I_2,\]\[ T_{f_4}T_{f_4}^*+T_{f_6}T_{f_6}^*+T_{f_7}T_{f_7}^*+T_{f_8}T_{f_8}^*=I_\sigma\otimes I_2.\]
Here $I_\pi$ is the identity of $\Hh_\pi$ for $\pi\in\hat{G}$ and $I_i$ the identity of $\Hh_i$ for $ i=1,2$.
Since
\[M_1M_2=\left[\begin{array}{ccc}1&1&2\\1&1&2\\2&2&4\end{array}\right]\]
and
\[T_{e_2}T_{f_4}, T_{f_3}T_{e_1}\in Hom(\iota, \iota\otimes\rho_1\otimes\rho_2),\]
\[T_{e_2}T_{f_6}, T_{f_3}T_{e_3}\in Hom(\epsilon, \iota\otimes\rho_1\otimes\rho_2),\]
\[T_{e_2}T_{f_7}, T_{e_2}T_{f_8}, T_{f_1}T_{e_2}, T_{f_3}T_{e_5}\in Hom(\sigma, \iota\otimes\rho_1\otimes\rho_2),\]
\[T_{e_4}T_{f_4}, T_{f_5}T_{e_1}\in Hom(\iota, \epsilon\otimes\rho_1\otimes\rho_2),\]
\[T_{e_4}T_{f_6}, T_{f_5}T_{e_3}\in Hom(\epsilon, \epsilon\otimes\rho_1\otimes\rho_2),\]
\[T_{e_4}T_{f_7}, T_{e_4}T_{f_8}, T_{f_2}T_{e_4}, T_{f_5}T_{e_5}\in Hom(\sigma, \epsilon\otimes\rho_1\otimes\rho_2),\]
\[T_{e_1}T_{f_1}, T_{e_5}T_{f_4}, T_{f_7}T_{e_1}, T_{f_8}T_{e_1}\in Hom(\iota, \sigma\otimes\rho_1\otimes\rho_2),\]
\[T_{e_3}T_{f_2}, T_{e_5}T_{f_6}, T_{f_7}T_{e_3}, T_{f_8}T_{e_3}\in Hom(\epsilon,\sigma\otimes\rho_1\otimes\rho_2),\]
\[T_{e_5}T_{f_7}, T_{e_5}T_{f_8}, T_{e_3}T_{f_5}, T_{e_1}T_{f_3}, T_{f_6}T_{e_4}, T_{f_4}T_{e_2}, T_{f_7}T_{e_5}, T_{f_8}T_{e_5}\in Hom(\sigma, \sigma\otimes\rho_1\otimes\rho_2),\] 
a possible choice of commuting squares is
\[e_2f_4=f_3e_1,\; e_2f_6=f_3e_3,\; e_2f_7=f_1e_2,\; e_2f_8=f_3e_5,\; e_4f_4=f_5e_1,\; e_4f_6=f_5e_3\]
\[e_4f_7= f_2e_4,\; e_4f_8=f_5e_5,\; e_1f_1=f_7e_1,\; e_5f_4=f_8e_1,\; e_3f_2=f_7e_3,\; e_5f_6=f_8e_3,\]
\[e_5f_7=f_6e_4,\; e_5f_8=f_4e_2,\; e_3f_5=f_7e_5,\; e_1f_3=f_8e_5.\]
\medskip
This data is enough to determine a rank $2$ graph $\Lambda$ associated to $\rho_1, \rho_2$.
But this is not the only choice, since for example we could have taken 
\[e_2f_4=f_3e_1,\; e_2f_6=f_3e_3,\; e_2f_8=f_1e_2,\; e_2f_7=f_3e_5,\; e_4f_4=f_5e_1,\; e_4f_6=f_5e_3\]
\[e_4f_8= f_2e_4,\; e_4f_7=f_5e_5,\; e_1f_1=f_7e_1,\; e_5f_4=f_8e_1,\; e_3f_2=f_8e_3,\; e_5f_6=f_7e_3,\]
\[e_5f_7=f_6e_4,\; e_5f_8=f_4e_2,\; e_3f_5=f_7e_5,\; e_1f_3=f_8e_5,\]
which will determine a different $2$-graph.

A direct analysis using the definitions shows that in each case, the $2$-graph $\Lambda$ is cofinal, it satisfies the  aperiodicity condition and every vertex connects to a loop with an entrance. It follows that   $C^*(\Lambda)$ is simple and purely infinite and  the Doplicher-Roberts algebra $\Oo_{\rho_1,\rho_2}$ is Morita equivalent with $C^*(\Lambda)$. 

The $K$-theory of $C^*(\Lambda)$ can be computed using Proposition 3.16 in \cite{E} and it does not depend on the choice of factorization rules. We have 
\[K_0(C^*(\Lambda))\cong\text{coker}[I-M_1^t\;\; I-M_2^t]\oplus\ker\left[\begin{array}{c}M_2^t-I\\I-M_1^t\end{array}\right]\cong \mathbb Z/2\mathbb Z,\]
\[K_1(C^*(\Lambda))\cong\ker[I-M_1^t\;\;I-M_2^t]/\text{im}\left[\begin{array}{c}M_2^t-I\\I-M_1^t\end{array}\right]\cong 0.\]
In particular, $\Oo_{\rho_1,\rho_2}\cong \Oo_3$. 

On the other hand, since $\rho_1, \rho_2$ are faithful, both $\Oo_{\rho_1}, \Oo_{\rho_2}$ are simple and purely infinite with
\[K_0(\Oo_{\rho_1})\cong \ZZ/2\ZZ,\; K_1(\Oo_{\rho_1})\cong 0,\; K_0(\Oo_{\rho_2})\cong \ZZ,\; K_1(\Oo_{\rho_2})\cong \ZZ,\]
so $\Oo_{\rho_1,\rho_2}\ncong \Oo_{\rho_1}\otimes \Oo_{\rho_2}$.

\end{example}
\bigskip

\begin{example}
With $G=S_3$  and $\rho_1=2\iota, \rho_2=\iota+\epsilon$, then $R=\{\iota, \epsilon\}$ so  $\Lambda$ will have two vertices and incidence matrices
\[M_1=\left[\begin{array}{cc}2&0\\0&2\end{array}\right],\;\; M_2=\left[\begin{array}{cc}1&1\\1&1\end{array}\right],\]
which give 
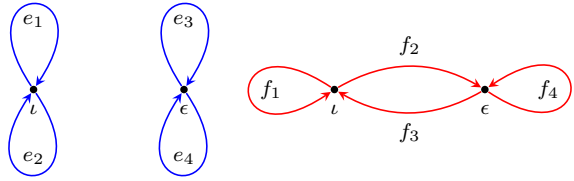
\[
\begin{tikzpicture}
\renewcommand{\ss}{\scriptstyle}
\node[inner sep=1.0pt, circle, fill=black]  (u) at (-3,0) {};
\node[below] at (u.south)  {$\ss \iota$};
\node[inner sep=1.0pt, circle, fill=black]  (v) at (-1,0) {};
\node[below] at (v.south)  {$\ss \epsilon$};
\node[inner sep=1.0pt, circle, fill=black]  (w) at (1,0) {};
\node[below] at (w.south)  {$\ss \iota$};
\node[inner sep=1.0pt, circle, fill=black]  (x) at (3,0) {};
\node[below] at (x.south)  {$\ss \epsilon$};
\draw[-stealth, semithick, blue] (u) .. controls (-4,1.5) and (-2, 1.5) ..   node[below,black] {$\ss e_1$} (u);
\draw[-stealth, semithick, blue] (u) .. controls (-2,-1.5) and (-4, -1.5) ..   node[above,black] {$\ss e_2$} (u);
\draw[-stealth, semithick, blue] (v) .. controls (-2,1.5) and (0, 1.5) ..   node[below,black] {$\ss e_3$} (v);

\draw[-stealth, semithick, blue] (v) .. controls (0,-1.5) and (-2, -1.5) ..   node[above,black] {$\ss e_4$} (v);

\draw[-stealth, semithick, red] (w) .. controls (-0.5,1) and (-0.5, -1) ..   node[right,black] {$\ss f_1$} (w);
\draw[-stealth, semithick, red] (w) to [out=30, in=150] node[above,black] {$\ss f_2$}  (x);
\draw[-stealth, semithick, red] (x) to [out=-150, in=-30] node[below,black] {$\ss f_3$}  (w);
\draw[-stealth, semithick, red] (x) .. controls (4.5,-1) and (4.5,1) ..   node[left,black] {$\ss f_4$} (x);

\end{tikzpicture}
\]

Again, a corresponding choice of isometric intertwiners  will determine some factorization rules, for example
\[e_1f_1=f_1e_2,\; e_2f_1=f_1e_1,\; e_1f_3=f_3e_3,\; e_2f_3=f_3e_4,\]
\[e_3f_2=f_2e_1,\; e_4f_2=f_2e_2,\; e_3f_4=f_4e_4,\; e_4f_4=f_4e_3.\]
Even though $\rho_1, \rho_2$ are not faithful, the obtained $2$-graph  is cofinal, satisfies the aperiodicity condition and every vertex connects to a loop with an entrance, so $\Oo_{\rho_1,\rho_2}$ is simple and purely infinite with trivial $K$-theory.
In particular, $\Oo_{\rho_1,\rho_2}\cong \Oo_2$.

Note that since $\rho_1, \rho_2$ have kernel $N=\la(123)\ra\cong \ZZ/3\ZZ$, we could replace $G$ by $G/N\cong \ZZ/2\ZZ$ and consider $\rho_1,\rho_2$ as  representations of $\ZZ/2\ZZ$.
\end{example}

\bigskip

\begin{example}
Consider $G=\ZZ/2\ZZ=\{0,1\}$ with $\hat{G}=\{\iota,\chi\}$ and character table
\begin{center}
\begin{tabular}{c|c|r}{}&$0$ &$1$\\
\hline
$\iota$& $1$&$1$\\
\hline
$\chi$ &$1$&$-1$\\

\hline

\end{tabular}
\end{center}
Choose the $2$-dimensional representations
\[\rho_1=\iota+\chi,\; \rho_2=2\iota,\; \rho_3=2\chi,\]
which determine a product system $\Ee$ such that $\Oo(\Ee)\cong \Oo_2\otimes\Oo_2\otimes \Oo_2$ and a Doplicher-Roberts algebra $\Oo_{\rho_1,\rho_2,\rho_3}\cong \Oo(\Ee)^{\ZZ/2\ZZ}$.

An easy computation shows that the incidence matrices of the blue, red and green graphs are
\[M_1=\left[\begin{array}{cc}1&1\\1&1\end{array}\right],\; M_2=\left[\begin{array}{cc}2&0\\0&2\end{array}\right],\; M_3=\left[\begin{array}{cc}0&2\\2&0\end{array}\right].\]

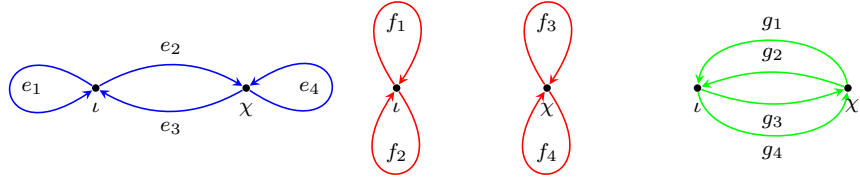
\[
\begin{tikzpicture}
\renewcommand{\ss}{\scriptstyle}
\node[inner sep=1.0pt, circle, fill=black]  (u) at (-5,0) {};
\node[below] at (u.south)  {$\ss \iota$};
\node[inner sep=1.0pt, circle, fill=black]  (v) at (-3,0) {};
\node[below] at (v.south)  {$\ss \chi$};
\node[inner sep=1.0pt, circle, fill=black]  (w) at (-1,0) {};
\node[below] at (w.south)  {$\ss \iota$};
\draw[-stealth, semithick, blue] (u) .. controls (-6.5,1) and (-6.5, -1) ..   node[right,black] {$\ss e_1$} (u);
\draw[-stealth, semithick, blue] (u) to [out=30, in=150] node[above,black] {$\ss e_2$}  (v);
\draw[-stealth, semithick, blue] (v) to [out=-150, in=-30] node[below,black] {$\ss e_3$}  (u);
\draw[-stealth, semithick, blue] (v) .. controls (-1.5,-1) and (-1.5,1) ..   node[left,black] {$\ss e_4$} (v);

\node[inner sep=1.0pt, circle, fill=black]  (x) at (1,0) {};
\node[inner sep=1.0pt, circle, fill=black]  (y) at (3,0) {};
\node[inner sep=1.0pt, circle, fill=black]  (z) at (5,0) {};

\draw[-stealth, semithick, red] (w) .. controls (-2,1.5) and (0, 1.5) ..   node[below,black] {$\ss f_1$} (w);
\draw[-stealth, semithick, red] (w) .. controls (0,-1.5) and (-2, -1.5) ..   node[above,black] {$\ss f_2$} (w);
\draw[-stealth, semithick, red] (x) .. controls (0,1.5) and (2, 1.5) ..   node[below,black] {$\ss f_3$} (x);

\draw[-stealth, semithick, red] (x) .. controls (2,-1.5) and (0, -1.5) ..   node[above,black] {$\ss f_4$} (x);

\draw[-stealth, semithick, green] (z) to [out=100, in=80] node[above,black] {$\ss g_1$} (y);
\draw[-stealth, semithick, green] (z) to [out=160, in=20] node[above,black] {$\ss g_2$} (y);
\draw[-stealth, semithick, green] (y) to [out=-80, in=-100] node[below,black] {$\ss g_4$} (z);
\draw[-stealth, semithick, green] (y) to [out=-20, in=-160] node[below,black] {$\ss g_3$} (z);
\node[below] at (x.south)  {$\ss \chi$};
\node[below] at (y.south)  {$\ss \iota$};
\node[below] at (z.east)  {$\ss \chi$};
\end{tikzpicture}
\]

With labels as in the figure, we  choose the following  factorization rules

\[e_1f_1=f_2e_1,\; e_1f_2=f_1e_1,\; e_2f_1=f_4e_2,\; e_2f_2=f_3e_2,\]
\[e_3f_3=f_2e_3,\; e_3f_4=f_1e_3,\; e_4f_4=f_3e_4,\; e_4f_3=f_4e_4,\]

\[f_1g_1=g_2f_3,\; f_1g_2=g_1f_3,\; f_2g_1=g_2f_4,\; f_2g_2=g_1f_4,\]
\[f_3g_3=g_4f_1,\; f_3g_4=g_3f_1,\; f_4g_3=g_4f_2,\; f_4g_4=g_3f_2,\]

\[e_1g_1=g_2e_4,\; e_1g_2=g_1e_4,\; e_2g_1=g_3e_3,\; e_2g_2=g_4e_3,\]
\[e_3g_3=g_1e_2,\; e_3g_4=g_2e_2,\; e_4g_3=g_4e_1,\; e_4g_4=g_3e_1.\]

\medskip

A tedious verification shows that  all the following paths are well defined
\[e_1f_1g_1,\; e_1f_1g_2,\; e_1f_2g_1, \; e_1f_2g_2,\; e_2f_1g_1,\; e_2f_1g_2,\; e_2f_2g_1,\; e_2f_2g_2,\]
\[e_3f_3g_3,\; e_3f_3g_4,\; e_3f_4g_3,\; e_3f_4g_4,\; e_4f_3g_3,\; e_4f_3g_4,\; e_4f_4g_3,\; e_4f_4g_4,\] 
so the associativity property is satisfied and we get a rank $3$ graph $\Lambda$ with $2$ vertices. It is not difficult to check that $\Lambda$ is cofinal, it satisfies the  aperiodicity condition and every vertex connects to a loop with an entrance, so $C^*(\Lambda)$ is simple and purely infinite.

Since $\partial_1=[I-M_1^t\; I-M_2^t\; I-M_3^t]:\ZZ^6\to \ZZ^2$ is surjective, using Corollary 3.18 in \cite{E}, we obtain
\[K_0(C^*(\Lambda))\cong \ker\partial_2/\text{im}\; \partial_3\cong 0,\;K_1(C^*(\Lambda))\cong \ker\partial_1/\text{im}\; \partial_2\oplus \ker\partial_3\cong 0,\]
where
\[\partial_2=\left[\begin{array}{ccc} M_2^t-I&M_3^t-I&0\\I-M_1^t&0&M_3^t-I\\0&I-M_1^t&I-M_2^t\end{array}\right],\;\;\partial_3=\left[\begin{array}{c}I-M_3^t\\M_2^t-I\\I-M_1^t\end{array}\right],\]
in particular $\Oo_{\rho_1,\rho_2,\rho_3}\cong\Oo_2$.


\end{example}

\bigskip
\begin{example}
Let $G=\TT$. We have $\hat{G}=\{\chi_k:k\in \ZZ\}$, where $\chi_k(z)=z^k$ and $\chi_k\otimes\chi_\ell=\chi_{k+\ell}$. The  faithful representations \[\rho_1=\chi_{-1}+\chi_0,\; \rho_2=\chi_0+\chi_1\] of $\TT$ will determine a product system $\Ee$ with $\Oo(\Ee)\cong \Oo_2\otimes \Oo_2$ and a Doplicher-Roberts algebra $\Oo_{\rho_1,\rho_2}\cong \Oo(\Ee)^\TT$ isomorphic to a corner in the $C^*$-algebra of a rank $2$ graph $\Lambda$ with $\Lambda^0=\hat{G}$ and infinite incidence matrices, where
\[M_1(\chi_k,\chi_\ell)=\begin{cases}1 &\text{if}\; \ell=k \;\text{or}\; \ell=k-1\\0& \text{otherwise,}\end{cases}\]
\[M_2(\chi_k,\chi_\ell)=\begin{cases} 1 &\text{if}\; \ell=k \;\text{or}\; \ell=k+1\\0& \text{otherwise.}\end{cases}\]
The skeleton  of $\Lambda$ looks like

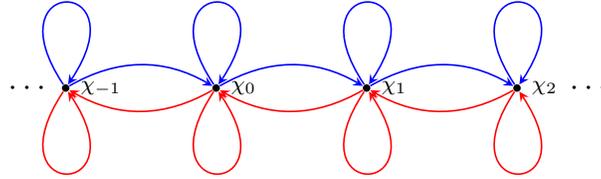
\[
\begin{tikzpicture}
\renewcommand{\ss}{\scriptstyle}
\node[inner sep=1.0pt, circle, fill=black]  (u) at (-3,0) {};
\node[right] at (u.east)  {$\ss \chi_{-1}$};
\node[left] at (u.west) {$ \cdots$};
\node[inner sep=1.0pt, circle, fill=black]  (v) at (-1,0) {};
\node[right] at (v.east)  {$\ss \chi_0$};
\node[inner sep=1.0pt, circle, fill=black]  (w) at (1,0) {};
\node[right] at (w.east)  {$\ss \chi_1$};
\node[inner sep=1.0pt, circle, fill=black]  (x) at (3,0) {};
\node[right] at (x.east)  {$\ss \chi_2$};
\node[right] at (x.east)  {$\hspace{5mm}\cdots$};
\draw[-stealth, semithick, blue] (u) .. controls (-4,1.5) and (-2, 1.5) ..   (u);
\draw[-stealth, semithick, red] (u) .. controls (-4,-1.5) and (-2, -1.5) ..   (u);
\draw[-stealth, semithick, blue] (v) .. controls (-2,1.5) and (0, 1.5) ..    (v);
\draw[-stealth, semithick, red] (v) .. controls (-2,-1.5) and (0, -1.5) ..  (v);
\draw[-stealth, semithick, blue] (w) .. controls (0,1.5) and (2, 1.5) ..    (w);
\draw[-stealth, semithick, red] (w) .. controls (0,-1.5) and (2, -1.5) ..   (w);
\draw[-stealth, semithick, blue] (x) .. controls (2,1.5) and (4, 1.5) ..    (x);
\draw[-stealth, semithick, red] (x) .. controls (2,-1.5) and (4, -1.5) ..   (x);
\draw[-stealth, semithick, blue] (u) to [out=30, in=150]  (v);
\draw[-stealth, semithick, red] (v) to [out=-150, in=-30]   (u);
\draw[-stealth, semithick, blue] (v) to [out=30, in=150]   (w);
\draw[-stealth, semithick, red] (w) to [out=-150, in=-30]   (v);
\draw[-stealth, semithick, blue] (w) to [out=30, in=150]   (x);
\draw[-stealth, semithick, red] (x) to [out=-150, in=-30]   (w);
\end{tikzpicture}
\]
and this $2$-graph is cofinal, satisfies the aperiodicity condition and every vertex connects to a loop with an entrance, so $C^*(\Lambda)$ is  simple and purely infinite.
\end{example}
\bigskip
\begin{example}
Let $G=SU(2)$. It is known (see p.84 in \cite{BD}) that the elements in $\hat{G}$ are labeled by $V_n$ for $n\ge 0$, where $V_0=\iota$ is the trivial representation on $\CC$, $V_1$ is the standard representation of $SU(2)$ on $\CC^2$, and for $n\ge 2$, $V_n=S^nV_1$, the $n$-th symmetric power. In fact, $\dim V_n=n+1$ and $V_n$ can be taken as the representation of $SU(2)$ on the space of homogeneous polynomials $p$ of degree $n$ in variables $z_1,z_2$, where for $\ds g=\left[\begin{array}{cc} a&b\\c&d\end{array}\right]\in SU(2)$ we have
\[(g\cdot p)(z)=p(az_1+cz_2, bz_1+dz_2).\]

The irreducible representations $V_n$ satisfy the Clebsch-Gordan formula
\[V_k\otimes V_\ell=\bigoplus_{j=0}^qV_{k+\ell-2j},\; q=\min\{k,l\}.\]
If we choose $\rho_1=V_1, \rho_2=V_2$, then we get a product system $\Ee$ with $\Oo(\Ee)\cong \Oo_2\otimes \Oo_3$ and a Doplicher-Roberts algebra $\Oo_{\rho_1,\rho_2}\cong \Oo(\Ee)^{SU(2)}$ isomorphic to a corner in the $C^*$-algebra of a rank $2$ graph with $\Lambda^0=\hat{G}$ and edges given by the matrices
\[M_1(V_k,V_\ell)=\begin{cases}1&\text{if}\; k=0\;\text{and}\; \ell=1\\ 1& \text{if}\; k\ge 1\;\text{and}\; \ell \in\{k-1,k+1\}\\0&\text{otherwise,}\end{cases}\]
\[M_2(V_k, V_\ell)=\begin{cases} 1 &\text{if}\; k=0\;\text{and}\; \ell=2\\1&\text{if}\; k=1\;\text{and}\; \ell\in  \{1,3\}\\ 1&\text{if}\; k\ge 2\;\text{and}\; \ell\in\{k-2,k,k+2\}\\0&\text{otherwise.}\end{cases}\]
The skeleton  looks like

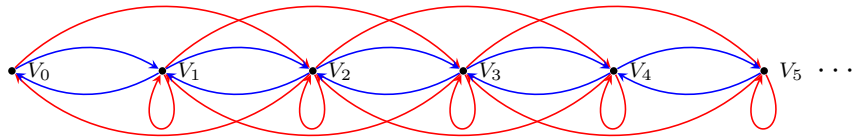
\[
\begin{tikzpicture}
\renewcommand{\ss}{\scriptstyle}
\node[inner sep=1.0pt, circle, fill=black]  (u) at (-5,0) {};
\node[right] at (u.east)  {$\ss V_0$};

\node[inner sep=1.0pt, circle, fill=black]  (v) at (-3,0) {};
\node[right] at (v.east)  {$\ss V_1$};
\node[inner sep=1.0pt, circle, fill=black]  (w) at (-1,0) {};
\node[right] at (w.east)  {$\ss V_2$};
\node[inner sep=1.0pt, circle, fill=black]  (x) at (1,0) {};
\node[right] at (x.east)  {$\ss V_3$};
\node[inner sep=1.0pt, circle, fill=black]  (y) at (3,0) {};
\node[right] at (y.east)  {$\ss V_4$};
\node[inner sep=1.0pt, circle, fill=black]  (z) at (5,0) {};
\node[right] at (z.east)  {$\ss V_5$};
\node[right] at (z.east)  {$\hspace{5mm}\cdots$};

\draw[-stealth, semithick, red] (v) .. controls (-2.5,-1) and (-3.5, -1) ..   (v);
\draw[-stealth, semithick, red] (w) .. controls (-0.5,-1) and (-1.5, -1) ..   (w);

\draw[-stealth, semithick, red] (x) .. controls (1.5,-1) and (0.5, -1) ..   (x);
\draw[-stealth, semithick, red] (y) .. controls (3.5,-1) and (2.5, -1) ..   (y);
\draw[-stealth, semithick, red] (z) .. controls (5.5,-1) and (4.5, -1) ..   (z);
\draw[-stealth, semithick, blue] (u) to [out=30, in=150]  (v);
\draw[-stealth, semithick, red] (u) to [out=45, in=135]  (w);
\draw[-stealth, semithick, red] (w) to [out=-135, in=-45]  (u);
\draw[-stealth, semithick, red] (v) to [out=45, in=135]  (x);
\draw[-stealth, semithick, red] (x) to [out=-135, in=-45]  (v);
\draw[-stealth, semithick, red] (w) to [out=45, in=135]  (y);
\draw[-stealth, semithick, red] (y) to [out=-135, in=-45]  (w);
\draw[-stealth, semithick, red] (x) to [out=45, in=135]  (z);
\draw[-stealth, semithick, red] (z) to [out=-135, in=-45]  (x);
\draw[-stealth, semithick, blue] (v) to [out=-150, in=-30]   (u);
\draw[-stealth, semithick, blue] (v) to [out=30, in=150]  (w);
\draw[-stealth, semithick, blue] (w) to [out=-150, in=-30]   (v);
\draw[-stealth, semithick, blue] (w) to [out=30, in=150]   (x);
\draw[-stealth, semithick, blue] (x) to [out=-150, in=-30]   (w);
\draw[-stealth, semithick, blue] (x) to [out=30, in=150]   (y);
\draw[-stealth, semithick, blue] (y) to [out=-150, in=-30]  (x);
\draw[-stealth, semithick, blue] (y) to [out=30, in=150]   (z);
\draw[-stealth, semithick, blue] (z) to [out=-150, in=-30]  (y);
\end{tikzpicture}
\]
and this $2$-graph is cofinal, satisfies the aperiodicity condition and every vertex connects to a loop with an entrance, in particular $\Oo_{\rho_1,\rho_2}$ is  simple and purely infinite.

\end{example}

\end{document}